\documentclass[a4paper,11pt]{amsart}
\usepackage{graphicx}
\usepackage[all]{xy}
\usepackage[english]{babel}

\usepackage{amsmath,amssymb}
\usepackage{enumerate}
\usepackage{amsthm}

\usepackage{amscd}

\usepackage{amsfonts}

\usepackage{mathrsfs}
\usepackage{epsfig}

\usepackage{stmaryrd}

\usepackage{layout}

\usepackage{fullpage}

\usepackage[usenames,dvipsnames]{color}

\usepackage[T1]{fontenc}

\usepackage[backref=page]{hyperref}

\DeclareMathOperator{\Pic}{Pic}

\DeclareMathOperator{\spec}{Spec}

\DeclareMathOperator{\Supp}{{Supp}}

\numberwithin{equation}{section}
\newcommand{\I}{\mathcal I}

\newcommand{\Z}{\mathbb Z}
\newcommand{\E}{\mathcal E}

\newcommand{\ov}{\overline}
\newcommand{\un}{\underline}
\renewcommand{\O}{\mathcal O}

\newtheorem{thm}{Theorem}[section]
\newtheorem{prop}[thm]{Proposition}

\newtheorem{question}[thm]{Question}

\newtheorem{fact}[thm]{Fact}

\theoremstyle{definition}
\newtheorem{defi}[thm]{Definition}

\theoremstyle{remark}
\newtheorem{rem}[thm]{Remark}
\newtheorem{convention}[thm]{}

\newcommand{\mgnb}{\overline{\mathcal M}_{g,n}}
\newcommand{\mgb}{\overline{\mathcal M}_{g}}

\newcommand{\Mn}{\mathfrak M_{g,n}}
\newcommand{\Mnx}{\mathfrak C_{g,n}}
\newcommand{\JM}{{\mathfrak J}_{\mathfrak M}}
\newcommand{\JMn}{\overline{\mathfrak J}_{g,n}}
\newcommand{\LMn}{\mathfrak J_{g,n}}

\newcommand{\JacMn}{\overline{\mathfrak Jac}_{g,n}}

\newcommand{\Nn}{\mathfrak N_{g,n}}

\begin{document}

 \title[{Universal N\'eron models for Jacobians of curves}]{Universal N\'eron models for Jacobians of curves with marked points}
\author{Margarida Melo}
\address{
Dipartimento di Matematica e Fisica, Universit\`a Roma Tre,
Largo San L. Murialdo 1 - 00146 Rome (Italy)} \email{melo@mat.uniroma3.it}
\address{CMUC and Mathematics Department of the University of Coimbra,
Apartado 3008, EC Santa Cruz 3001 - 501 Coimbra (Portugal)} \email{mmelo@mat.uc.pt}

\begin{abstract}
We consider the problem of constructing universal N\'eron models for families of curves with sections. By applying a construction of the author of universal compactified Jacobians over the moduli stack of reduced curves with markings and a result by J. Kass, we get a positive answer for smooth families of curves with planar singularities over Dedekind schemes.
\end{abstract}


\maketitle


\section{Introduction}

Let $f:X\to B$ be a family of curves over a Dedekind scheme $B$ and assume that there is a dense open subscheme $U\subseteq B$ of $B$ such that $X_U$ is smooth over $U$. Let $J(X_U)$ be the relative Jacobian of $X_U\to U$.  The problem of understanding which would be good models for $J(X_U)$ over $B$ is of course very classical and there have been different approaches to attack it along the years.
One of these is the so-called N\'eron model, constructed by Andr\'e N\'eron in \cite{neron}; it is a smooth and separated model for $J(X_U)$ over $B$, uniquely determined by the so-called \textit{N\'eron mapping property} (see \ref{S:Neron}). Actually, N\'eron's original work  applies more generally to models of abelian varieties and in the case of Jacobians the situation is better understood thanks to work of Raynaud in \cite{raynaud}, who described the N\'eron model of $J(X_U)$ as the biggest separated quotient of the relative Jacobian of the family (see \ref{S:neronJacobians}).

The N\'eron model is therefore something intrinsic to the family so it is natural to ask the following questions:
\begin{itemize}
\item Do N\'eron models glue together as the families of curves change?
\item Is there a natural compactification of the N\'eron model and do these compactifications glue together as the families change?
\item Is there a modular description of the N\'eron model (and of its compactification)?
\end{itemize}

Let us focus on the case of families of reduced curves with marked points. Then the above questions could be better reformulated as follows.

\begin{question}\label{Q:int}
Let $\Mn$ be the moduli stack of reduced curves of given genus $g$  and with $n$ distinct markings.
Is there a modular algebraic stack $\Nn$ endowed with a map onto $\Mn$ such that $\Nn$ is a universal N\'eron model for the Jacobians of families of curves in $\Mn$?
Moreover, can we describe a modular compactification of $\Nn$ over $\Mn$?
\end{question}

By  a \textit{universal N\'eron model} we mean an algebraic stack $\Nn$ endowed with a map $\pi$ onto $\Mn$ such that given a family of curves $(f:X\to B; \sigma_1,\dots, \sigma_n)\in \Mn(B)$ over a Dedekind scheme $B$, which is smooth over a dense open subset $U\subseteq B$, the pullback of $\Nn\to\Mn$ via the moduli map of the family $\mu_f:B\to\Mn$ is isomorphic to the N\'eron model of $J(X_U)$ over $B$:
\begin{equation*}
\xymatrix{
{N(J(X_U))\cong \mu_f^*(\Nn)}\ar[r] \ar[d]& {\Nn}\ar[d]_{\pi}\\
{B}\ar[r]_{\mu_f} & {\Mn}.
}
\end{equation*}

A different possibility for extending $J(X_U)$ over $B$ is to take a compactified Jacobian of the family $f$. This is a proper model of $J(X_U)\to U$ and in the case when $f$ has non-integral fibers, there are a number of different possibilities for constructing such an object.
It is of course natural to try to understand the relation between compactified Jacobians and N\'eron models.
In particular, there are universal modular compactifications of Jacobians so it is expected that an answer to Question \ref{Q:int} above can be obtained via one of these universal compactified Jacobians.

In this note, we will show that the universal compactified Jacobians constructed by the author in \cite{melo2} (following the constructions by Esteves in \cite{est1}) do indeed give a positive answer to the above question in the case of families of reduced curves with at least one section and such that the total space of the family $X$ is regular.
More precisely, given a polarization on $\Mn$, i.e., a vector bundle $\E$ with integral slope in its universal family $\Mnx$, (see Definition \ref{D:polstack}),  and a section $\sigma$ of $\Mnx\to \Mn$, one can define the moduli stack $\JMn^{\E,\sigma}$ parametrizing families of reduced marked curves together with torsion-free rank-1 simple sheaves which are $(\E,\sigma)$-quasistable (see \ref{C:stability} and \ref{C:stacknotation}). This is an algebraic stack endowed with a proper and representable morphism onto $\Mn$. Consider also $\LMn^{\E,\sigma}$ to be the substack of $\JMn^{\E,\sigma}$ parametrizing invertible sheaves. Then in Theorem \ref{T:main} below we show that we can take the universal N\'eron model $\Nn$ to be equal to $\LMn^{\E,\sigma}$, for any  $\E$ and $\sigma$. A modular compactification for $\Nn$ is then naturally given by taking $\JMn^{\E,\sigma}$.

\subsection{Previous results}

The relation between N\'eron models of Jacobians of smooth curves and compactified Jacobians of theirs degenerations is, as we mentioned above, quite natural to investigate, and therefore there has been some work trying to understand different aspects of it, including universal versions of the problem as the one we formulate in Question \ref{Q:int}.

Consider the case of a family of curves $f:X\to B$ over a Dedekind scheme $B$ such that $f$ is smooth over a dense open subset $U\subset B$. Then one can consider a compactified Jacobian $\ov J$ of the family, which will be proper but possibly singular, and take the locus of line bundles $J$ inside $\ov J$. The question is to understand the relation between $J$ and $N(J(X_U))$. Actually, in the literature, N\'eron models have also been used to understand the relation among different compactifications of the Jacobians: according to the terminology introduced by Caporaso in \cite{capNT}, a compactification  is said to be of {\it N\'eron type} if $J$ and $N(J(X_U))$ are isomorphic.
The most general result in this direction so far is due to Jesse Kass who showed in \cite{kass} that $J$ is actually isomorphic to $N(J(X_U))$ in the case when the total space of the family $X$ is smooth and if $\ov J$ is either any fine and proper compactified Jacobian of Esteves or any Simpson's compactified Jacobian associated to a non-degenerate polarization (see Fact \ref{F:kass} and Remark \ref{R:simpson}).
In the case of nodal curves, the question was already considered by Oda and Seshadri in \cite{OS} with respect to their compactification.
In the case of Esteves' compactification, Kass result was previously known for nodal curves by work of \cite{Busonero}  and by work of the author together with Viviani in \cite{MV} for compactified Jacobians associated to non-degenerate polarizations, both for Esteves' and Simpson's compactifications.
The techniques used by these last authors are essentially of combinatorial nature and  differ from the ones by Kass, whose results apply only when the moduli problem is fine.

A universal approach to the problem as we formulated in Question \ref{Q:int} was first studied by Lucia Caporaso in \cite{capN} for certain degrees and by the author in \cite{meloZ} in general in the case of Caporaso's \textit{Balanced Picard stack of degree $d$} $\ov{\mathcal P}_{d,g}$ over the moduli stack of stable curves $\ov{\mathcal M}_g$. See \ref{S:cap} for a detailed account on the relation of these results and the results obtained in the present paper.

More recently, in \cite{holmes}, David Holmes considered the problem of constructing a universal N\'eron model for the moduli stack of stable curves with marked points $\ov{\mathcal M}_{g,n}$, i.e., a partial compactification of the universal Jacobian over $\ov{\mathcal M}_{g,n}$ satisfying the usual properties of N\'eron models, in particular admitting a group structure and satisfying a suitable universal property.
In this situation, the author actually shows that such an object can not exist over $\ov{\mathcal M}_{g,n}$ but rather over a universal base-change $\widetilde{\mathcal M}_{g,n}\to\ov{\mathcal M}_{g,n}$ of it (see \ref{S:holmes}).

A result of the same flavour was recently obtained by Alessandro Chiodo in \cite{chiodo}, who constructed a group scheme $\Pic_g^{0,l}$ over the Abramovich-Vistoli moduli stack of $l$-twisted stable curves $\ov{\mathcal M}_g^l$ with the property that $\Pic_g^{0,l}$ is a universal N\'eron model for curves in $\ov{\mathcal M}_g$ (see \ref{S:chiodo}).

We are not able to endow our universal compactified Jacobian stacks with a group structure as Holmes and Chiodo do, however it does have the property of yielding universal N\'eron models for families of curves  with worse singularities rather then just nodal. Moreover, it is indeed defined over $\Mn$ rather than over a modification of it, as it is the case in Holmes' and Chiodo's constructions as we mentioned above.

\section{Preliminaries and notation}

We work over an algebraically closed field $k$ (of arbitrary characteristic).


\subsection{Curves}

\begin{convention}\label{C:curves}
A \textbf{curve} (over $k$) is  a reduced projective scheme over $k$ of pure dimension $1$, assumed to be connected.

Given a scheme $T$, a \textbf{family of curves} is a flat and projective morphism $f:X\to T$ whose geometric fibers are curves.
\end{convention}

\begin{convention}
A \textbf{subcurve} $Z$ of a curve $X$ is a closed $k$-scheme $Z \subseteq X$ that is reduced  and of pure dimension $1$.  We say that a subcurve $Z\subseteq X$ is non-trivial if $Z\neq \emptyset, X$.
 \end{convention}

\begin{convention}
A curve $X$ is called \textbf{Gorenstein} if its dualizing sheaf $\omega_X$ is a line bundle.

Given a subcurve $Z\subseteq X$ of a Gorenstein curve, we denote by $w_Z$ the degree of $\omega_X$ in $Z$, i.e., $w_Z:=\deg_Z\omega_X$.
\end{convention}

\begin{convention}\label{N:Jac-gen}
Given a curve $X$, the \textbf{generalized Jacobian} of $X$, denoted by $J(X)$,
is an algebraic group whose $k$-valued points are line bundles on $X$ of multidegree $\un 0$ (i.e. having
degree $0$ on each irreducible component of $X$) together with the multiplication given by the tensor product.
\end{convention}

\begin{convention}\label{C:relativeJac}
Let $f:X\to T$ be a family of curves.
Consider the functor
\begin{equation}\label{E:func-Jbar}
{\mathbb J}_f : \{{\rm Schemes}/T\}  \to \{{\rm Grp}\}
\end{equation}
defined to be the \'etale sheafification of the functor which associates to a $T$-scheme $Y$ the set $\Pic(Y_T)$ of isomorphism classes of line bundles on $Y_T:=X\times _T Y$.

Under our hypothesis, the functor $\mathbb J_f$ is representable by an algebraic space (\cite[Thm. 7.3]{artin}) that we will denote with $J_f$ and that is called the relative Jacobian of the family. The fibers of $J_f$ are representable by smooth group schemes locally of finite type and $J_f$ itself is always relatively finitely presented and formally smooth.
\end{convention}


\subsection{Sheaves}\label{simple sheaves}

\begin{convention}
A coherent sheaf $I$ on a curve $X$ is said to be:
\begin{enumerate}[(i)]
\item \emph{rank-$1$} if $I$ has generic rank $1$ at every irreducible component of $X$;
\item \emph{torsion-free} if $\Supp(I)=X$ and every non-zero subsheaf $I'\subseteq I$ is such that $\dim \Supp(I')=1$;
\item \emph{simple} if ${\rm End}_k(I)=k$.
\end{enumerate}

\noindent Note that if $L$ is a line bundle on a curve $X$ then it is clearly simple, rank-$1$ and torsion-free.
\end{convention}

 \begin{convention}
 Let $f:X\to T$ be a family of curves and $I$ a $T$-flat coherent sheaf on $X$.
We say that $I$ is relatively torsion-free (resp. rank 1, resp. simple) over $T$ if the fiber of $I$ over $t$, $I(t)$, is torsion-free
(resp. rank 1, resp. simple) for every geometric point $t$ of $T$.
\end{convention}

\begin{convention}
Given a rank-$1$ coherent sheaf $I$ on $X$, its degree $\deg(I)$ is defined by $\deg(I):=\chi(I)-\chi(\O_X)$, where $\chi(I)$ (resp. $\chi(\O_X)$) denotes the Euler-Poincar\'e characteristic of $I$ (resp. of the trivial sheaf $\O_X$).

\noindent For a subcurve $Y$ of $X$ and a torsion-free sheaf $I$ on $X$, the restriction $I_{|Y}$ of $I$ to $Y$ is not necessarily a torsion-free sheaf on $Y$; denote by $I_Y$ the maximum torsion-free quotient of $I_{|Y}$. Notice that $I_Y$ is torsion-free rank-$1$.
We let $\deg_Y (I)$ denote the degree of $I_Y$ on $Y$, that is, $$\deg_Y(I) := \chi(I_Y )-\chi(\O_Y).$$
\end{convention}

\subsection{N\'eron models}\label{S:Neron}
Let $B$ be a Dedekind scheme and $A\to U$ a family of Abelian varieties over a dense open subscheme $U\subseteq B$.
We denote by $N(A)$ the N\'eron model of $A$ over $B$,  which is a separated $B$-smooth group scheme $N(A)\to B$, extending $A\to U$. Recall that $N(A)$ is uniquely determined by the following universal property (the N\'eron mapping property): every $U$-morphism $u: Z_U\to A$ defined on some scheme $Z$ smooth over $B$ admits a unique extension to a $B$-morphism $u: Z \to N(A)$. (see \cite{neron} and \cite[Def. 1]{BLR}).

In what follows we will be particularly interested in the case when $A=J$ is the relative Jacobian of a family of curves $f:X\to U$. We will in this case frequently abuse terminology and refer to $N(J)$ as the N\'eron model of the family of curves rather then the N\'eron model of the Jacobian.

\section{Fine compactified Jacobians for families of curves}

In what follows we will describe the construction due to Esteves in \cite{est1} and based on work by Altman-Kleiman in \cite{AK} of fine compactifications of the relative Jacobians of families of reduced curves.

\subsection{Altman-Kleiman's algebraic space of simple sheaves}

Let $f:X\to T$ be a family of curves (recall that our families are always assumed to have reduced fibers (see \ref{C:curves} above)).
Consider the functor
\begin{equation}\label{E:func-Jbar}
\ov{\mathbb J}_f : \{{\rm Schemes}/T\}  \to \{{\rm Sets}\}
\end{equation}
defined to be the \'etale sheafification of the functor which associates to a $T$-scheme $Y$ the set of isomorphism classes of $T$-flat, coherent sheaves on $Y_T:=X\times _T Y$
which are relatively simple rank-$1$ torsion-free sheaves.

\begin{fact}\label{F:AK-simple}{(\cite{AK},\cite{est1})}
The functor $\ov{\mathbb J}_f$ is represented by an algebraic space $\ov{J}_f$, locally of finite type over $T$ and it satisfies the existence part of the valuative criterion of properness.
\end{fact}

The terminology \textit{fine compactification} is due to the fact that $\ov{\mathbb J}_f$ is indeed representable and not simply coarsely representable. In particular, up to taking an \'etale base change, there exists a universal sheaf $\I$ on $X \times \ov J_f$, uniquely determined up to tensor product by an invertible sheaf on the base scheme.

The algebraic space is clearly not of finite type. Indeed, we have a decomposition
$$\ov{J}_f=\coprod_{d\in \mathbb Z}\ov J_f^d,$$
where each $\ov J_f^d$ corresponds to the sub-algebraic space of $\ov J_f$ parametrizing simple sheaves of total degree $d$.
Moreover, in the case when $f$ has fibers which are not integral curves, fixing the total degree is not sufficient and we immediately see that the algebraic spaces $\ov J_f^d$ are neither of finite type nor separated over $T$. In order to get finite type and separated quotients of each $\ov J_f^d$ one has to fix a stability condition on the sheaves and therefore  bound the possible solutions. This has been worked out by Esteves, as we now explain.

\subsection{Semistable simple sheaves}\label{S:semistable}

We start by defining a polarization on a single curve $X$, following \cite[1.2]{est1}.

\begin{defi}\label{D:d-pol}
Let $X$ be a curve of genus $g$. A $d$-\textit{polarization} on $X$ (or shortly a \textit{polarization}) is a vector bundle $E$ on $X$ of rank $r>0$ and degree $r(d-g+1)$.
\end{defi}

Given a polarization $E$ on a curve $X$, Esteves defines in \cite[1.2, 1.4]{est1} the notions of (semi-)stability with respect to $E$ and, moreover, by further fixing a point $p$ on $X$, a notion of $(E,p)$-quasistability, which stands in between semi-stability and stability, as we now explain.

\begin{defi}\label{C:stability}
Fix a $d$-polarization $E$ on a curve $X$ and a smooth point $p\in X$.
A torsion-free rank-$1$ sheaf $I$ on $X$ with $\deg(I)=d$ is said to be
\begin{enumerate}[(i)]
\item $E$-semi-stable or semi-stable with respect to $E$ if for every proper subcurve $\emptyset\neq Y\subsetneq X$, we have
\begin{equation}\label{E:stabCond}
\chi(I_Y)\geq\frac{\deg E_Y}{r};
\end{equation}
\item $E$-stable or stable with respect to $E$ if inequality \eqref{E:stabCond}
holds strictly for every proper subcurve $\emptyset\neq Y\subsetneq X$;
\item $(E,p)$-quasistable, or $p$-quasistable with respect to $E$ if it is semi-stable and if strict inequality holds above whenever $p\in Y$.
\end{enumerate}
\end{defi}

From the definition, it is clear that in order to give a polarization on a curve $X$ it suffices to give a tuple of rational numbers $\un q=\{\un q_{C_i}\}$, one for each irreducible component $C_i$ of $X$, such that $|\un q|:=\sum_i \un q_{C_i}\in \Z$ (see e.g. \cite{MV}). However, since we want to deal with families, it is more convenient to deal with the vector bundles.

\begin{defi}\label{D:non-deg}
A polarization $E$ is said to be \textit{non-degenerate} if for all proper subcurves $Y\subset X$, $\frac{\deg E_Y}r\notin \mathbb Z$, which implies that  inequality \eqref{E:stabCond} is never an equality.
\end{defi}

Notice that, for a non-degenerate polarization $E$ on $X$, the notions of stability, semistability and quasistabily clearly coincide.

\begin{rem}
If we assume that $X$ is Gorenstein, then inequality \eqref{E:stabCond} above is easily seen to be equivalent to the following more usual inequality
$$\deg_Y(I)\geq \frac{\deg E_Y}{r}+\frac{w_Y}2-\frac{|Y\cap \overline{X\setminus Y}|}2,$$ where $|Y\cap \overline{X\setminus Y}|$ denotes the length of the scheme-theoretic intersection of $Y$ with its complementary curve $\overline{X\setminus Y}$.
\end{rem}

\subsubsection{}Let now $f:X\to S$ be a family of curves. Fix a \textit{relative $d$-polarization} for the family, i.e., a vector bundle $E$ on $X$ of rank $r>0$ and relative degree $r(d-g+1)$ over $T$. Fix also a section $\sigma:T\to X$ of $f$ through the $T$-smooth locus of $X$.

\begin{defi}
\begin{enumerate}[(i)]
\item A torsion-free, rank-$1$ sheaf $I$ on $X$ over $T$ is relatively stable
(resp. semistable) with respect to a polarization $E$ over $T$ if $I(t)$ is stable
(resp. semistable) with respect to $E(t)$ for every geometric
point $t$ of $T$.
\item A torsion-free, rank-$1$ sheaf $I$ on $X$ over $T$ is relatively $\sigma$-quasistable
with respect to $E$ over $T$, or simply $(E,\sigma)$-quasistable if $I(t)$ is $\sigma(t)$-quasistable with respect to $E(t)$ for every
geometric point $t$ of $T$.
\end{enumerate}
\end{defi}

Denote by $\ov{\mathbb J}_f^{E,ss}$ (resp. $\ov{\mathbb J}_f^{E,s}$, resp. $\ov{\mathbb J}_f^{E, \sigma}$) the subfunctors of $\ov{\mathbb J}_f$ parametrizing (relatively) $E$-semistable (resp. $E$-stable, resp. $(E, \sigma)$-quasistable) sheaves. Then we have the following result due to Esteves  (see \cite[Theorem A]{est1}):

\begin{fact}\label{F:esteves}
The moduli functors of relatively $E$-semistable (resp. $E$-stable, resp. $(E, \sigma)$-quasistable) simple torsion-free $T$-flat sheaves over $X$,   $\ov{\mathbb J}_f^{E,ss}$ (resp. resp. $\ov{\mathbb J}_f^{E,s}$, resp. $\ov{\mathbb J}_f^{E, \sigma}$), are representable by algebraic spaces $\ov J_f^{E,ss}$ (resp. $\ov J_f^{E,s}$, resp. $\ov J_f^{E, \sigma}$) of finite type over $T$. Moreover,
\begin{enumerate}[(i)]
\item $\ov J_f^{E,ss}$ is universally closed over $T$;
\item $\ov J_f^{E,s}$ is separated over $T$;
\item $\ov J_f^{E, \sigma}$ is proper over $T$.
\end{enumerate}
\end{fact}
Clearly, we always have inclusions  $\ov J_f^{E,s}\subseteq  \ov J_f^{E,\sigma}\subseteq  \ov J_f^{E,ss}$. If the polarization we start with is non-degenerate (i.e., if the restriction of $E$ to any fiber of $f$ is non-degenerate), then all these schemes actually coincide.

Notice that the ``universal sheaf'' for torsion-free rank-$1$ simple sheaves $\I$ on $X\times \ov J_f$ descends to give universal sheaves for each of the above quotients of $\ov J_f$.

\begin{rem}
It follows from \cite{est1} that if $T=k$ is an algebraically closed field, then $\overline J_f^{E, \sigma}$ is actually a projective scheme.
\end{rem}

\section{(Compactified) Jacobians and N\'eron models}

Let $B$ be a Dedekind scheme over an algebraically closed field $k$ and $f:X\to B$ be a family of curves which is smooth over a dense open subset $U\subseteq B$.

\subsection{N\'eron models of Jacobians}\label{S:neronJacobians}

Consider the relative Jacobian of the family $$J_{f_{|U}}:=J(X_U)\to U,$$ which is a family of Abelian varieties over $U$, and  consider its N\'eron model over $B$, $N(J(X_U))$.
This special situation has been studied in detail by Raynaud in  \cite{raynaud}, who related $N(J(X_U))$ with the relative Jacobian of the whole family $J_f$ defined in \ref{C:relativeJac}.
Recall that the relative Jacobian is locally finitely presented and formally smooth  over the base, however it fails to be separated as long as there are reducible fibers in $f$.

Inside $\mathbb J_f$, consider the subfunctor $\mathbb E_f$ defined to be the scheme-theoretic closure of the identity section. Then $\mathbb E_f$ is representable by an \'etale $T$-group space that we will denote by $E_f$. Then $Q_f:=J_f/E_f$ is the greatest separated quotient of $J_f$ and it coincides with the N\'eron model of $J(X_U)$ over $B$ (see \cite[Theorem 4 in 9.5]{BLR}).

A more concrete description of $N(J(X_U))$ has been given by Caporaso in \cite[Lemma 3.10]{capN} using twisters of the special fiber $X_0$ of $f$.

\subsection{Line bundles loci in compactified Jacobians and N\'eron models}
Let  $\ov J\to B$ be a relative compactified Jacobian of the family $f$. By this we mean a proper algebraic space $\ov J\to B$ which coincides with the relative Jacobian when restricted to the generic fiber.

There is a vast literature trying to relate compactified Jacobains for the family $f$ and the N\'eron model $N(J(X_U))$. In particular, one wonders if a certain locus inside $\ov J$ (e.g., the smooth locus of the morphism $\ov J\to B$) coincides with $N(J(X_U))$.

The following result, due to J. Kass, relates the locus of line bundles inside Esteves compactified Jacobians with $N(J(X_U))$ in the case when the total space of the family $X$ is smooth, and is probably the most general result on the subject obtained so far.

\begin{fact}{\cite[Theorem 1]{kass}}\label{F:kass}
Let $f:X\to B$ be a family of curves with smooth total space $X$ and generic fiber $X_\eta$.
Let $E$ be a polarization on $f$ and $\sigma:B\to X$ a section of $f$. Denote with $J_f^{E,\sigma}$ the locus of line bundles on $\ov J_f^{E,\sigma}$. Then $J_f^{E,\sigma}$ is isomorphic to $N(J(X_\eta))$, the N\'eron model of the generic fiber of the family over $B$.
\end{fact}

\begin{rem}\label{R:simpson}
Fact \ref{F:kass} above is also proved in loc. cit. to hold for a Simpson compactified Jacobian whenever it is associated to a non-degenerate polarization (see Definiton \ref{D:non-deg}).
Notice that this last property implies that stability and semistability coincide and therefore that Simpson's compactified Jacobian is also a fine moduli space in this case.
\end{rem}

\begin{rem}
The fact that the locus of line bundles of fine compactified Jacobians for different polarizations is isomorphic to the N\'eron model of the family does not imply that the compactifications themselves are isomorphic. Indeed, an example of different non-degenerate polarizations over nodal curves yielding non-isomorphic compactified Jacobians was recently exhibited in \cite[Example 3.11]{MRV}.
\end{rem}

In the special case when the family $f$ has only nodal singularities, the previous result was known by work of the author together with Viviani in \cite{MV}. In loc. cit., the authors also describe a stratification of the whole compactification $\ov J_f^{E,\sigma}$ in terms of N\'eron models of partial normalizations of the special fiber of $f$.
In this case, the fiber of $N(J(X_U))$ over the special fiber of $B$ is known to be isomorphic to a disjoint sum of copies of the generalized Jacobian of the special fiber of $f$, so the result acquires a combinatorial flavor (indeed N\'eron models are in this case determined by the so-called degree class group, which is defined using the dual graph of the special fiber of $f$ and that has many interesting incarnations in the literature: see for instance the introduction of \cite{BMS} for an account of these different appearances).

\section{Universal fine compactifictified Jacobians over curves with marked points}

For any non-negative integer $n$, let $\Mn$ be the moduli stack parametrizing $n$-marked reduced curves of given genus $g$. Sections of $\Mn$ over a scheme $T$ consist of families $f:X\to T$ of reduced curves, together with $n$ disjoint sections $\sigma_1,\dots, \sigma_n:T \to X $ which factor through the smooth locus of the family. Recall that $\Mn$ is algebraic (see e.g. \cite[Cor. 1.5]{hall}).

Consider now the stack $\JacMn$ whose $T$-sections consist of pairs
$(\xymatrix{
X \ar[r]_{f} & T \ar @/_/[l]_{\sigma_i}
}
, F)$
where
\begin{itemize}
\item $(f:X\to T;\sigma_1,\dots,\sigma_n)\in\Mn(T)$;
\item $F$ is a $T$-flat coherent sheaf over $X$ whose fibers over $T$ are torsion-free rank-1 simple sheaves.
\end{itemize}
Morphisms between two such pairs
$(\xymatrix{
X \ar[r]_{f} & T \ar @/_/[l]_{\sigma_i}
}
, F)$
and
$(\xymatrix{
X' \ar[r]_{f'} & T' \ar @/_/[l]_{\sigma'_i}
}
, F')$
are given by Cartesian diagrams
\begin{equation}\label{D:morphism}
\xymatrix{
{X} \ar[r]^{\bar g} \ar[d]^{f} & {X'} \ar[d]_{f'}\\
{T} \ar[r]_{g}  \ar @{->} @/^/[u]^{\sigma_i} & {T'} \ar @{->} @/_/[u]_{\sigma'_{i}}
}
\end{equation}
commuting with the sections, together with an isomorphism
$\beta: F\to\bar g^* F'$.

There is clearly an action by $\mathbb G_m$ on $\JacMn$ and we define $\JMn$ to be the $\mathbb G_m$-rigidification of $\JacMn$.

Denote by $\pi$ the natural forgetful morphism
$$\pi:\JMn\to \Mn$$
given by forgetting the sheaf.

The following is Theorem A in \cite{melo2}:

\begin{fact}\label{F:Th1}
The stack $\JMn$ is representable over $\Mn$. In particular, $\JMn$ (and also $\JacMn$) is algebraic.
\end{fact}

Let $(f:X\to T; \sigma_1,\dots, \sigma_n)\in\Mn(T)$ be a family of marked curves and let $\mu_f:T\to \Mn$ be the corresponding modular map on $\Mn$.
Denote by $\mu_f^*(\JMn)$  the base-change of $\pi:\JMn\to \Mn$ via $\mu_f$, as indicated in the diagram below.
\begin{equation*}
\xymatrix{
{\mu_f^*(\JMn)} \ar[r]\ar[d] & {\JMn} \ar[d]_{\pi}\\
{T} \ar[r]_{\mu_f} & {\Mn}
}
\end{equation*}
 Moreover, from the proof of Theorem A in \cite{melo2}, we have.

\begin{prop}\label{P:basechange1}
Notation as above. Then $\mu_f^*(\JMn)$ is representable by the fine moduli space $\overline J_f$ of torsion-free rank-1 simple sheaves for the family $f:X\to T$ introduced in Fact \ref{F:AK-simple} above.
\end{prop}

\begin{proof}
The fact that $\mu_f^*(\JMn)$ is representable by an algebraic space is a direct consequence of Fact \ref{F:Th1} above.
Now, in the proof of Theorem A in \cite{melo2}, we observed that the category fibered in groupoids associated to the functor of torsion-free rank-$1$ simple sheaves $\ov{\mathbb J}_f$ is equivalent to $\mu_f^*(\JMn)$. This is enough to conclude since from \ref{F:AK-simple} we know that $\ov{\mathbb J}_f$ is represented by $\ov J_f$.
\end{proof}

Let $\Mnx\to \Mn$ be the natural forgetful map from the universal family $\Mnx$ onto $\Mn$.
The following is Definition 3.7 in \cite{melo2}.

\begin{defi}\label{D:polstack}
Let $\E$ be a vector bundle on $\Mnx$.
Then $\E$ is a $d$-polarization on $\Mn$ if it has positive rank $r$ and degree $r(d-g+1)$.
\end{defi}
In fact, given a family of marked curves $(f:X\to T,\sigma_1\dots,\sigma_n)$, the pullback of $\E$ via $\overline \mu_f$, the base-change of $\mu_f$ via $\Mnx\to\Mn$,
\begin{equation*}
\xymatrix{
{X} \ar[r]^{\overline\mu_f} \ar[d]_{f} & {\Mnx} \ar[d]\\
{T} \ar[r]_{\mu_f} & {\Mn}
}
\end{equation*}
is a vector bundle $\E_f:=\overline\mu_f^*(\E)$, which turns out to be a $d$-polarization on the family $f$.

In the same way, given a section $\sigma:\Mn\to \Mnx$ of $\pi$, its pullback via $\mu_f$ gives a section $\sigma_f$ of the family $f:X\to T$.

\begin{convention}\label{C:stacknotation}
Let $\E$ be a $d$-polarization on $\Mn$ and $\sigma:\Mn\to \Mnx$ a section of the forgetful morphism $\Mnx\to \Mn$. We will denote by  $\JMn^{\E,\sigma}$ the open substack of $\JMn$ parametrizing pairs
$$(\xymatrix{
X \ar[r]_{f} & T \ar @/_/[l]_{\sigma_i}
}
, F)$$
such that $F$ is $\sigma_f$-quasistable with respect to $\E_f$.
Denote also by $\JMn^{\E,s}$ (resp. $\JMn^{\E,ss}$) the open substacks of $\JMn$ parametrizing pairs as above where the sheaf $F$ is $\E_f$-stable (resp. semistable).
\end{convention}

Then we have

\begin{fact}{\cite[Theorem B]{melo2}}
\begin{enumerate}[(i)]
\item The stack $\JMn^{\E,ss}$ is universally closed over $\Mn$.
\item The stack $\JMn^{\E,s}$ is separated over $\Mn$.
\item The stack $\JMn^{\E,\sigma}$ is proper over  $\Mn$.
\end{enumerate}
\end{fact}

\begin{rem}
It follows from Propositions 3.6 and  4.4 in \cite{melo2} that the restriction of $\JMn^{\E,\sigma}$ to the stack $\Mn^{lp}$ of curves with locally planar singularities is smooth and irreducible and that its restriction to the stack $\overline{\mathcal M}_{g,n}$ of stable curves is moreover a Deligne-Mumford stack.
\end{rem}

Let us briefly discuss how to specify a number of different $d$-polarizations on $\Mn$.
By taking $\E:=\mathcal O_{\Mnx}^{\oplus r}$, for any $r\in \mathbb Z$, we get a polarization of degree $g-1$. In order to get polarizations of different degrees we can use the sections $\sigma_1,\dots,\sigma_n:\Mn\to \Mnx$.
Then, for any combination of integers $a_i\in\mathbb Z$ summing up to $r(d-g+1)$, we get that
$$\E:=\mathcal O_{\Mnx}(\sum_{i=1}^n a_i\sigma_i)\oplus \mathcal O_{\Mnx}^{\oplus r-1}$$
is a $d$-polarization in $\Mn$.
In the case we are working over the stack $\Mn^{lp}$ then, since we are dealing with Gorenstein curves, it also makes sense to take the relative dualizing sheaf of the universal family morphism, $\Mnx\to\Mn$, which we denote by  $\omega$. Then we can also consider the so-called canonical polarization for each degree $d$, $$\E^{can}:=\omega^{\otimes(d-g+1)}\oplus \mathcal O_{\Mnx^{lp}}^{\oplus 2g-3}.$$

Now, as above, consider the pullback $\mu_f*(\JMn^{\E,\sigma})$ of $\JMn^{\E,\sigma}\to\Mn$ via $\mu_f:T\to\Mn$. Then, analogously to the situation described in Proposition \ref{P:basechange1} above, we have.

\begin{prop}\label{P:basechange2}
Notation as above. Then $\mu_f^*(\JMn^{\E,\sigma})$ is representable by the fine moduli space $\overline J_f^{\E_f,\sigma_f}$ of $(\E_f,\sigma_f)$-quasistable  torsion-free rank-1 simple sheaves for the family $f:X\to T$ introduced in Fact \ref{F:esteves} above.
\end{prop}

\begin{proof}
This is a direct consequence of Proposition \ref{P:basechange1} above along with the modular descriptions of $\JMn^{\E,\sigma}$ and of $\overline J_f^{\E_f,\sigma_f}$.
\end{proof}

\section{Universal compactified Jacobians and N\'eron models}


Let $\E$ be a polarization on $\Mn$ and $\sigma:\Mn\to\Mnx$ a section of the forgetful morphism $\Mnx\to \Mn$.
Denote by $\LMn^{\E,\sigma}$ the substack of $\JMn^{\E,\sigma}$ parametrizing invertible sheaves.

Let $(f:X\to T; \sigma_1,\dots, \sigma_n)$ be a family of marked curves. As in the previous section, denote by $\E_f$ and $\sigma_f$ the induced polarization and section, respectively, of $(\E,\sigma)$ on $f$ via pullback by the moduli morphism $\mu_f:T\to \Mn$.

Denote by $J_f^{\E_f,\sigma_f}$ the locus inside the Esteves compactified Jacobian $\ov J_{f}^{\E_f,\sigma_f}$ parametrizing line bundles.
Then the following is an immediate consequence of Proposition \ref{P:basechange2} above.

\begin{prop}\label{P:smoothbasechange}
Notation as above. Then $\mu_f^*(\LMn^{\E,\sigma})$ is representable by $J_f^{\E_f,\sigma_f}$.
\end{prop}

We can now state the main result of the present paper.

\begin{thm}\label{T:main}
Let $B$ be a Dedekind scheme and $(f:X\to B;\sigma_1,\dots,\sigma_n)$ a family of marked curves with smooth total space $X$ and generic fiber $X_\eta$.
Then we have
$$N(J(X_\eta))\cong \mu_f^*(\LMn^{\E,\sigma})$$
where $\E$ is any polarization on $\Mn$ and $\sigma$ any section of $\Mnx\to\Mn$.
\end{thm}

\begin{proof}

The result follows directly by combining the statements of Proposition \ref{P:smoothbasechange} with Fact \ref{F:kass}.

\end{proof}

Some comments are due.

\begin{rem}
The role of the marked points on Theorem \ref{T:main} is to guarantee the existence of a section $\sigma:\Mn\to\Mnx$: the statement holds as it is if $n\geq 1$. In particular, we get universal N\'eron models for families of curves admitting a section.
\end{rem}

\begin{rem}
For $n=0$ and in the special case when $g.c.d.(d-g+1,2g-2)=1$ the result holds for families of Deligne-Mumford stable curves and for the so-called canonical polarization (in which case the restrictions to $\mgnb$ of the stacks $\JMn^{\E,ss}$, $\JMn^{\E,s}$ and $\JMn^{\E,\sigma}$ all coincide). However, in this situation, our result coincides with the one of Caporaso in \cite{capN}.  We will describe this result in detail in \ref{S:cap}.
\end{rem}

\begin{rem}
Since we get universal N\'eron models for families of reduced curves over Dedekind schemes whose total space is smooth, the curves are indeed forced to live in a smooth surface, hence its singularities are necessarily locally planar. Therefore, our results regarding N\'eron models could be restated by replacing the stack of reduced curves by theirs substack of locally planar curves. We chose to work with the stack of reduced curves so that the reader could have the most general statement regarding the existence of fine compactified universal Jacobian stacks, but in what concerns applications to N\'eron models it really would make no difference. In \ref{S:holmes} we will discuss the results by Holmes, which indeed include the case of families whose total space is not necessarily smooth, even though the singularities of the fibers are assumed to be nodal.
\end{rem}

\subsection{Relation with Caporaso's results}\label{S:cap}

Fix now $n=0$  and consider $\ov{\mathcal M}_g\subset \mathfrak M_g$ to be the  stack parametrizing Deligne-Mumford stable curves of genus $g$.
Let $\ov{\mathcal P}_{d,g}$ be the balanced Picard stack constructed by Caporaso in \cite{capN}.
The stack $\ov{\mathcal P}_{d,g}$ is a universal compactification of the degree $d$ Jacobian variety over $\ov{\mathcal M}_g$ and it is endowed with a morphism onto $\ov{\mathcal M}_g$ which is representable if and only if the fixed degree $d$ and genus $g$ satisfy the relation $g.c.d.(d-g+1,2g-2)=1$.

It is easy to see that, if this numerical condition is satisfied, $\ov{\mathcal P}_{d,g}$ coincides with the restriction of our compactification $\ov{\mathfrak J}_g^{\E,\sigma}$ to $\ov{\mathcal M}_g$ in the case when the polarization $\E$ is given by the canonical polarization $\E^{can}:=\omega^{\otimes(d-g+1)}\oplus \mathcal O_{\ov{\mathcal C}_{g,1}}^{\oplus 2g-3}$, where with $\omega$ we denote the relative dualizing bundle of the universal family morphism $\ov{\mathcal C}_{g}\to\ov{\mathcal M}_g$.

Indeed, in this special situation, our Theorem \ref{T:main} coincides with   in \cite{capN}. A generalization of this result for any $d$ but restricting to a special loci of $\ov{\mathcal M}_{g}$ (over which $\E^{can}$ is a non-degenerate polarization) was also given by the author in \cite{meloZ}. In fact, when $g.c.d.(d-g+1,2g-2)\neq 1$, there are curves for which $\E^{can}$ is degenerate, which implies that there are strictly semi-stable sheaves in the compactification. There are therefore further identifications in the GIT quotient, which ends up possibly not containing the N\'eron model.

\subsection{Relation with the work by Holmes and Chiodo}

\subsubsection{}\label{S:holmes}
In \cite{holmes}, the author considers the following problem: is it possible to construct a partial compactification of the universal Jacobian stack over $\mgnb$ which is separated, admits a group structure prolonging that on the Jacobian and yielding universal N\'eron models for families of Jacobians of curves?

While it seems to be impossible to get all such properties, the author takes a different approach, namely he proves that there exists a stack $N_{g,n}$ over a base-change $\widetilde{\mathcal M}_{g,n}$ of $\mgnb$ which is separated, admits a group structure and satisfies a universal property. This object gives, as in our case, universal N\'eron models for the universal Jacobian over stable marked curves  in the following sense.
For any $g$ and $n$ such that $2g-2+n>0$, let  $(f:X\to T;\sigma_1,\dots,\sigma_n)$ be a family of marked curves over a dedekind scheme such that $f$ is \textit{weakly transversal to the boundary}, which is a weaker condition then asking that the total space of the family is regular (see \cite[Sec. 12]{holmes}). Then the modular map $\mu_f: T\to \mgnb$ of the family of curves  factors through $\widetilde{\mathcal M}_{g,n}$, and the pullback of $N_{g,n}$ over such map is the N\'eron model of $f$.
This means that Holmes' partial compactification can be seen as a universal N\'eron model for families of Jacobians, but it is has the drawback of existing only over a base-change of $\mgnb$.

\subsubsection{}\label{S:chiodo}
A similar approach has been taken by Chiodo in \cite{chiodo}. In this case, the author works over the Abramovich-Vistoli moduli stack of $l$-twisted curves $\mgb^l$.
Over this stack, Chiodo constructs a group scheme $\Pic_g^{0,l}$ extending the pullback of the universal Jacobian. As in the case of Holmes, this group scheme also yields universal N\'eron models for Jacobians  in the following sense. Consider a family of curves $f:X\to \spec R$, over the spectrum of a discrete valuation ring $R$ such that the total space of the family $X$ is smooth. Then, after extracting  the modular map $\mu_f:\spec R \to\mgb$ lifts to a morphism $\spec R_l\to \mgb^l$, when $R_l$ is the discrete valuation ring obtained by extracting an $l$th root from the uniformizer of $R$ and the pullback of $\Pic_g^{0,l}$ over this map is again isomorphic to the N\'eron model of the base-change of the original family to $\spec R_l$.

For both Holmes and Chiodo's contructions, a natural question one might ask is which is the relation between the universal N\'eron model obtained in each case and the pullback of our universal compactified Jacobian stack $\JM^{\E,\sigma}$ to $\widetilde{\mathcal M}_{g,n}$ and $\mgb^l$, respectively.
In this direction, the author was informed by David Holmes that his PhD student Giulio Orecchia has shown that given an aligned family of nodal curves $X\to T$ (see \cite{holmes} for the definition of aligned families) which is relatively smooth over a dense open subset $U\subseteq T$ and with total space $X$ regular, then the line bundle locus of Esteves' compactified jacobian is exactly the N\'eron model of the relative Jacobian of $X_U\to U$.
As a consequence, any universal compactified Jacobian $\ov{\mathfrak J}_{\widetilde C_{g,n}}^{\E,\sigma}$ of a minimal desingularization $\widetilde C_{g,n}$ of the pulback of the universal curve $\overline{\mathcal C}_{g,n}$ over $\widetilde{\mathcal M}_{g,n}$ yields a relative compactification of the universal N\'eron model over $\widetilde{\mathcal M}_{g,n}$.

\section*{Acknowledgements}

The author wishes to thank Alessandro Chiodo, David Holmes and specially Jesse Kass for reading and commenting on a preliminary version of the paper.  The author also thanks the referee for his/her careful reading of the paper and for the suggestions he/she presented.

This work was funded by a Rita Levi Montalcini Grant, funded by the Italian government through MIUR.
The author is a member of CMUC (Center for Mathematics of the University of Coimbra) -- UID/MAT/00324/2013, funded by the Portuguese Government through FCT/MEC and co-funded by the European Regional Development Fund through the Partnership Agreement PT2020.

\end{document}